\documentclass[12pt]{amsart}
\usepackage{color}
\usepackage[all]{xy}
\usepackage{amssymb,amsmath}

\usepackage{mathrsfs}
\usepackage{eucal}

\usepackage{setspace}

\date{August~31, 2014}

\calclayout
\newtheorem{dummy}{anything}[section]
\newtheorem{theorem}[dummy]{Theorem}
\newtheorem*{thma}{Theorem A}
\newtheorem*{thmb}{Corollary B}
\newtheorem*{thmc}{Theorem C}
\newtheorem{lemma}[dummy]{Lemma}
\newtheorem{proposition}[dummy]{Proposition}
\newtheorem{corollary}[dummy]{Corollary}

\theoremstyle{definition}
\newtheorem{definition}[dummy]{Definition}

  \newtheorem{example}[dummy]{Example}
  
  \newtheorem{remark}[dummy]{Remark}

    \newtheorem*{question}{Question}
  \newtheorem*{acknowledgement}{Acknowledgement}

\newcommand
{\eqncount}{\setcounter{equation}{\value{dummy}}%
\addtocounter{dummy}{1}}

\newcommand{\cE}{\mathcal E}
\newcommand{\cF}{\mathcal F}
\newcommand{\cG}{\mathcal G}

\newcommand{\cK}{\mathcal K}

\newcommand{\bC}{\mathbf C}
\newcommand{\bD}{\mathbf D}

\newcommand{\bZ}{\mathbb Z}

\newcommand{\bbZ}{\mathbb Z}
\DeclareMathOperator{\Hom}{Hom}
 
\DeclareMathOperator{\Image}{im}
 \DeclareMathOperator{\Ext}{Ext}
 \DeclareMathOperator{\Ind}{Ind}
 \DeclareMathOperator{\Inc}{Inc}
\DeclareMathOperator{\Res}{Res}
\DeclareMathOperator{\Iso}{\mathfrak{Iso}}
\DeclareMathOperator{\Supp}{Supp}
\DeclareMathOperator{\Mor}{Mor}
 \DeclareMathOperator{\Map}{Map}

\DeclareMathOperator{\ind}{Ind}
\DeclareMathOperator{\res}{Res}
 \DeclareMathOperator{\Dim}{Dim}
\DeclareMathOperator{\HomDim}{HomDim}
\DeclareMathOperator{\hdim}{hdim}
  \DeclareMathOperator{\rk}{rank}

\newcommand{\snor}{\lhd}

\DeclareMathOperator{\Mod}{Mod}
\DeclareMathOperator{\Def}{Def} 
\newcommand{\RMod}{R\textrm{-Mod}}
 \DeclareMathOperator{\Inf}{Inf}

\newcommand{\cy}[1]{\bZ/{#1}}
\newcommand{\la}{\langle}
\newcommand{\ra}{\rangle}
\newcommand{\vv}{\, | \,}
\newcommand{\bd}{\partial}
\newcommand{\id}{\mathrm{id}}

\def\bZp{\bZ_{(p)}}

\def\G{\varGamma}
\newcommand{\maprt}[1]{\xrightarrow{#1}}

\DeclareMathOperator{\Or}{Or}
\newcommand{\Sub}{{\mathcal S}(G)}
\newcommand\RG{R\G}
\newcommand\ZG{\bZ\G}
\newcommand\OrG{\Or _{\cF}G}
\newcommand\oK{\overline{K}}
\newcommand\oG{\overline{G}}
\newcommand\un{\underline}

\newcommand\uR[1]{R[{#1}^{\, \textbf{?}\,}]}
\newcommand\uRb[1]{R[{#1}\hphantom{{\hskip-4pt}l}^{\textbf{?}\,}]}
\newcommand\uC[1]{\CC({#1}^{\textbf{?}};R)}
\newcommand\uCZ[1]{\CC({#1}^{\textbf{?}};\bZ)}
\DeclareMathOperator{\hd}{hdim}
\newcommand\CC{\bC}
\newcommand\DD{\bD}

\newcommand\uH[1]{H_*({#1}^{\textbf{?}};R)}
\DeclareMathOperator{\Qd}{Qd}
\newcommand{\bigast}{\divideontimes}

\begin{document}

\title{Homotopy Representations over the Orbit Category}
\author{Ian Hambleton}
\author{Erg\"un Yal\c c\i n}

\address{Department of Mathematics, McMaster University,
Hamilton, Ontario L8S 4K1, Canada}

\email{hambleton@mcmaster.ca }

\address{Department of Mathematics, Bilkent University,
06800 Bilkent, Ankara, Turkey}

\email{yalcine@fen.bilkent.edu.tr }

\thanks{Research partially supported by NSERC Discovery Grant A4000. The second author is partially supported by T\" UB\. ITAK-TBAG/110T712.}

\begin{abstract} Let $G$ be a finite group. The unit sphere in a finite-dimensional orthogonal $G$-representation motivates the definition of homotopy representations, due to tom Dieck. We introduce an algebraic analogue and establish its basic properties, including the Borel-Smith conditions and realization by finite $G$-CW-complexes.
\end{abstract}

\maketitle

\section{Introduction}
\label{sect:introduction}

Let $G$ be a finite group. The unit spheres $S(V)$ in  finite-dimensional orthogonal representations of $G$ provide the basic examples of smooth $G$-actions on spheres. Moreover, character theory reveals intricate relations between the dimensions of the fixed subspheres $S(V)^H$, for subgroups $H \leq G$, and the structure of the isotopy subgroups $\{G_x \vv x \in S(V)\}$. Our goal is to  better understand  the constraints on  these basic invariants,  in order to construct new smooth \emph{non-linear} finite group actions on spheres (see \cite{hpy1}, \cite{h-yalcin2}). 

In order to put this problem in a more general setting, tom Dieck \cite[II.10.1]{tomDieck2}
introduced \emph{geometric homotopy representations}, as finite $G$-CW-complexes $X$ with the property that each fixed set $X^H$ is homotopy equivalent to a sphere. In this paper, we study an algebraic version of this notion for $R$-module chain complexes over the orbit category $\G _G =\Or _{\cF} G$, with respect to a ring $R$ and a family $\cF$ of subgroups of $G$. We usually work with $R = \bZp$, for some prime $p$, or $R=\bbZ$. This theory was developed by L\"uck \cite[\S 9, \S 17]{lueck3} and tom Dieck \cite[\S 10-11]{tomDieck2}.
 
The homological dimensions of the various fixed sets are encoded in a conjugation-invariant function  
$\un{n} \colon \Sub \to \bbZ$, where $\Sub$ denotes the set of subgroups of $G$. The function $\un{n}$ is \emph{supported on the family $\cF$},  if  $\un{n}(H) = -1 $ for $H \notin \cF$ (see Definition \ref{defn:superclass}). 
We say that a finite projective chain complex $\CC$ over $\RG_G$ is an \emph{$R$-homology $\un{n}$-sphere} if 
 the reduced homology of $\CC (K)$ is the same as the reduced homology of an $\un{n}(K)$-sphere (with coefficients in $R$) for all $K \in \cF$.

If $\CC$ is an $R$-homology $\un{n}$-sphere, which satisfies the internal homological conditions observed for representation spheres (see Definition \ref{def:algrep}), then we say that $\CC$ is an  \emph{algebraic homotopy representation}. By \cite[II.10]{tomDieck2}, these conditions are all necessary for $\CC$ to be chain homotopy equivalent to a geometric homotopy representation. In Proposition \ref{prop:tightness conditions},  we show more generally that these conditions hold for $\CC$ an $R$-homology $\un{n}$-sphere, whenever $\un{n}=\Dim\CC$, where $\Dim \CC$ denotes the chain dimension function of $\CC$. When this equality holds, we say that $\CC$ is a \emph{tight} complex.
 
In general, $\un{n}(K) \leq \Dim \CC(K)$ for each $K \in \cF$, and one would expect obstructions to finding a tight complex which is chain homotopy equivalent to a given $R$-homology $\un{n}$-sphere. Our first main result shows the relevance of the internal homological conditions for this question. 
 
\begin{thma}\label{thm:mainthm} Let  $G$ be a finite group, and  $\cF$ be a family of subgroups of $G$.  Suppose that 
 \begin{enumerate}
\item $R$ is a principal ideal domain,
\item $\un{n}\colon \Sub \to \bZ$ is a conjugation-invariant function supported on $\cF$, and
\item $\CC$ is a finite chain complex of free $\RG_G$-modules which is an
$R$-homology $\un{n}$-sphere.
\end{enumerate}
Then $\CC$ is chain homotopy equivalent to a finite
free chain complex $\DD$ satisfying $\un{n}=\Dim \DD$ if and only if $\CC$ is an algebraic homotopy representation. 
\end{thma}
  
Theorem A was motivated by \cite[Theorem 8.10]{hpy1}, which states that a finite chain complex of free $\bbZ\G_G$-modules can be realized by a geometric $G$-CW-complex if it is a tight homology $\un{n}$-sphere such that $\un{n} (H) \geq 3$ for all $H \in \cF$.  Upon combining these two statements, we get the following geometric realization result.

\begin{thmb}\label{cor:maincor}
Let $\CC$ be a finite  chain complex of free $\ZG_G$-modules which is a
homology $\un{n}$-sphere. 
If $\CC$ is an  algebraic homotopy representation, 
and  in addition, if $\un{n}
(K) \geq 3$ for all $K \in \cF$, then there is a finite $G$-CW-complex $X$, with isotropy in $\cF$, such that $\uCZ{X}$ is chain homotopy equivalent to
$\CC$ as chain complexes of $\ZG_G$-modules. 
\end{thmb}

We are  interested in constructing finite $G$-CW-complexes with some restrictions on the family of isotropy subgroups. We say a $G$-CW-complex $X$ has \emph{rank one isotropy} if for every $x\in X$, the isotropy subgroup $G_x$ has $\rk G_x \leq 1$. Recall that \emph{rank} of a finite group $G$ is defined as the largest integer $k$ such that $(\bbZ /p)^k \leq G$ for some prime $p$. We will use Theorem A and Corollary B to study the following: 

\begin{question} Which finite groups $G$ admit a finite $G$-CW-complex $X$ 
with rank one isotropy, such that $X$ is homotopy equivalent to a sphere~?   
\end{question}

One motivation for this work is that rank one isotropy examples lead to  free $G$-CW-complex actions of  finite groups on \emph{products} of spheres (see  Adem and Smith \cite{adem-smith1}).

In \cite{hpy1} we gave the first non-trivial example, by constructing  a finite $G$-CW-complex $X \simeq S^n$ for the symmetric group $G=S_5$, with cyclic $2$-group isotropy. However, the arguments used special features of the isotropy family. 
 Corollary B now provides an effective general method for the geometric realization of algebraic models. The  algebraic homotopy representation conditions are easy to check locally over $R= \bZp$ at each prime, and fit well with the local-to-global procedure for constructing chain complexes $\CC$ over $\ZG_G$. In a sequel \cite{h-yalcin2} to this paper, we apply Corollary B to construct infinitely many new examples  with rank one isotropy, 
 for certain interesting families of rank two groups.

In Section \ref{sect:Borel-Smith} we consider the algebraic version of a well-known theorem in transformation groups:  the  dimension function of a homotopy representation satisfies certain conditions called the Borel-Smith conditions (see Definition \ref{def:Borel-Smith}). 

\begin{thmc}\label{thm:IntroBorel-Smith}
Let $G$ be a finite group, 
$R=\bZ/p$, and $\cF$ be a given family of subgroups of $G$. If $\CC$ is a finite projective chain complex over $\RG _G$, which is an $R$-homology
$\un{n}$-sphere, then the function $\un{n}$ satisfies the Borel-Smith conditions at the prime $p$.  
\end{thmc}

As an application, we show that such a finite projective chain complex over $\RG_G$ does not exist for the group $G=\Qd(p)$ with respect to the family $\cF$ of rank 1 subgroups
(see Example \ref{ex:Qdp} and Proposition \ref{pro:algQdp}). This is an important group theoretic constraint  on the existence question for geometric homotopy representations with rank one isotropy.

One of the main ideas in the proof of Theorem C is the reduction of a given chain complex of $R\G_G$-module $\CC$ to a chain complex over $R\G_{K/N}$ for a subquotient $K/N$ 
appearing in the Borel-Smith conditions. For this reduction, we introduce \emph{inflation} and \emph{deflation} of modules over the orbit category, via restriction and induction associated to a certain functor $F$ (see Section \ref{sec:deflation}). Then we use spectral sequence arguments to conclude that the conditions given in the Borel-Smith conditions hold for these reduced chain complexes over $R\G _{K/N}$.

 Here is a brief outline of the paper. 
In Section \ref{sect:AlgHomRep} we give the precise setting and background definitions for
the concepts just presented (see Definition \ref{def:algrep}) and prove the  ``only if" direction of Theorem A. The ``if" direction of Theorem A is proved in Section \ref{sec:thma}, together with Corollary B. In Section \ref{sect:Borel-Smith}  we discuss the Borel-Smith conditions and prove Theorem C.

Our methods involve the study of finite-dimensional chain complexes of finitely generated projective modules over the orbit category, called \emph{finite projective chain complexes}, for short. Such chain complexes are the algebraic analogue of finitely-dominated $G$-CW complexes.

\begin{acknowledgement}  We thank the referees for their helpful comments and suggestions.
\end{acknowledgement}

\section{Algebraic homotopy representations}
\label{sect:AlgHomRep}

Let $G$ be a finite group and $\cF$ be a family of subgroups of $G$ which is closed under conjugations and taking subgroups. 
The orbit category $\OrG$ is defined as the category whose objects are orbits of type
$G/K$, with $K \in \cF$, and where the morphisms from $G/K$ to $G/L$
are given by $G$-maps:
$$ \Mor _{\OrG} (G/K , G/L ) = \Map _G (G/K, G/L ).$$
The category $\G _G= \OrG$ is a small category, and we can consider the
module category over $\G _G$ in the following sense. Let $R$ be a commutative ring with 1.
A \emph{(right) $\RG _G$-module} $M$ is a contravariant functor from $\G _G$ to
the category of $R$-modules. We denote the $R$-module $M(G/K)$
simply by $M(K)$ and write $M(f)\colon  M(L) \to M(K)$  for a $G$-map $f\colon 
G/K \to G/L$.  

The category of $\RG _G$-modules is an abelian category, so the usual concepts of homological algebra, such as kernel, direct sum,
exactness, projective module, etc., exist for $\RG_G$-modules.
A sequence of $\RG _G$-modules $0 \to A \to B \to C \to 0$
is \emph{exact} if and only if $$ 0 \to A(K) \to B(K) \to C(K) \to 0$$ is
an exact sequence of $R$-modules for every $K \in \cF$. For an $\RG_G$-module $M$ the $R$-module $M(K)$ can also be considered as an $RW_G(K)$-module in an obvious way where $W_G(K) = N_G(K)/K$. We will follow the convention in \cite{lueck3} and  consider $M(K)$ as a right $RW_G(K)$-module. In particular, we will consider the sequence above as an exact sequence of right $RW_G(K)$-modules.

For each $H \in \cF$, let $F_H:=R[G/H ^?]$ denote the $R\G_G$-module with values $F_H (K)=R[(G/H)^K]$ for every $K \in \cF$, and where for every  $G$-map $f: G/L \to G/K$, the induced map $F_H (f): R[(G/H)^K]\to R[(G/H)^L]$ is defined in the obvious way. By the Yoneda lemma, there is an isomorphism $$\Hom _{R\G_G} (R[G/H^?], M)\cong M(H)$$ for every $R\G_G$-module $M$. From this it is easy to show that the module $R[G/H^?]$ is a projective module in the usual sense, for each $H \in  \cF$. An $R\G_G$-module is called \emph{free} if it is isomorphic to a direct sum of $R\G_G$-modules of the form $R[G/H^?]$. It can be shown that an $R\G_G$-module is projective if and only if it is a direct summand of a free module. The further details about the properties of modules over the orbit category can be found in \cite{hpy1} (see also L\"uck \cite[\S 9,\S 17]{lueck3} and tom Dieck \cite[\S 10-11]{tomDieck2}).   

In this section we consider chain complexes $\CC$ of $\RG _G$-modules, with respect to a given family $\cF$. When we say a chain complex we always mean a non-negative complex, so $\CC _i=0$ for $i<0$. We call a chain complex $\CC$ \emph{projective} (resp.~\emph{free}) if for all $i\geq 0$,
the modules $\CC_i$ are projective (resp.~free). We say that a chain complex $\CC$ is \emph{finite} if $\CC_i = 0$ for all $i > n$, for some $n\geq 0$, and the chain modules $\CC_i$ are all finitely generated $\RG_G$-modules. 

\begin{remark} Up to chain homotopy equivalence, there is no difference between finite projective 
chain complexes and finite-dimensional projective chain complexes with  finitely generated homology (see \cite[3.6]{h-yalcin2}). For this reason, our definitions and results are mostly stated for \emph{finite} chain complexes. 
\end{remark}

We define the \emph{support} of a chain complex $\CC$ over $R\G_G$ as the family of subgroups $$\Supp(\CC) = \{ H \in \cF \vv \CC(H) \neq 0\}.$$

It is sometimes convenient to vary the family of subgroups. 

\begin{definition}
If $\cF \subset \cG$ are two families, the orbit category $\G _{G, \cF}=\Or_\cF G$ is a full-subcategory of $\Gamma _{G, \cG}=\Or_\cG G$. If $M$ is a module over $R\G_{G, \cF}$, then we define
$\Inc_\cF^\cG(M)(H) =M(H)$, if $H \in \cF$, and zero  otherwise. 
Similarly, 
for a module $N$ over $R\G_{G, \cG}$,  define $\res_\cF^\cG(N)(H) = N(H)$, for $H \in \cF$. We extend to maps and chain complexes similarly. Note that $\Supp(\Inc_\cF^\cG(\CC)) = \Supp(\CC)$, and $\Supp(\res_\cF^\cG(\DD)) = \Supp(\DD) \cap \cF$.  
\end{definition}

 Given a $G$-CW-complex $X$, 
there is an associated chain complex of $\RG_G$-modules over the family of all subgroups
$$\uC{X}:   \quad  \cdots \to \uR{X_n}
\xrightarrow{\bd_{n}}  \uR{X_{n-1}} \rightarrow
\cdots\xrightarrow{\bd_{1}}  \uR{X_0} \to 0$$ where $X_i$ denotes
the set of (oriented) $i$-dimensional cells in $X$ and $\uR{X_i}$ is the
$\RG_G$-module defined by $ \uR{X_i}(H)= R[X_i^H]$  for every $H \leq G$. We denote the
homology of this complex by $\uH{X}$. The chain complex $\CC (X^H; R)$ is actually defined for all subgroups $H \leq G$, but for a given family of subgroups $\cF$, we can restrict its values from $\Or(G)$ to the full sub-category $\OrG$. 

The smallest family containing all the isotropy subgroups  $\{G_x \vv x \in X\}$ is 
$$\Iso(X) = \{ H \leq G \vv X^H \neq \emptyset\}$$
and this motivates our notion of support for algebraic chain complexes.
In particular, we have 
$$ \Supp(\res_\cF(\uC{X})) = \cF \cap \Iso(X). $$  
If the family $\cF$ includes all the isotropy subgroups of $X$, then the complex $\uC{X}$ is a chain complex of free $\RG_G$-modules, hence projective $\RG_G$-modules, but otherwise the chain modules $\uR{X_n}$ may not be projective over $\RG _G$.

Given a finite-dimensional $G$-CW-complex $X$,  there is a
\emph{dimension function} $$\Dim X \colon  \Sub  \to \bZ,$$  given by $(\Dim X )(H)=\dim X^H$ for all $H \in \Sub$ where $\Sub$ denote the set of all subgroups of $G$. By convention,  we set $\dim \emptyset = -1$ for the dimension of the empty set.
In a similar way, we define the following.

\begin{definition}\label{defn:Dim} The (chain) \emph{dimension function} of a finite-dimensional chain complex $\CC$ over $\RG_G$ is defined as the function $\Dim \CC \colon  \Sub \to \bZ$ which has the value
 $$(\Dim \CC )(H)=\dim \CC (H)$$ for all 
 $H \in \cF$, 
 where the \emph{dimension} of a chain complex of $R$-modules is defined as the largest integer $d$ such $C_d\neq 0$ (hence the zero complex has dimension $-1$). If 
$H \notin \cF$, then we set $(\Dim \CC)(H)=-1$. 
\end{definition}

The dimension function $\Dim \CC \colon \Sub \to \bZ$ is  \emph{conjugation-invariant}, meaning that it takes the same value on conjugate subgroups of $G$. The term \emph{super class function} is often used for such functions. 

\begin{definition}\label{defn:superclass}
 The \emph{support} of a super class function $\un{n}$ is defined as the set 
$$\Supp (\un{n})=\{ H \leq G \colon \un{n} (H) \neq -1 \}.$$
We say that a super class function $\un{n}\colon \Sub  \to \bZ$ is \emph{supported on $\cF$}, if  $\Supp(\un{n}) \subseteq \cF$. Note that $\Supp (\CC)\subseteq \cF$ is the support of the dimension function $\Dim\CC$ of a chain complex $\CC$ over $\RG_G$. 
\end{definition}

In a similar way, we can define the \emph{homological dimension function} of a chain complex $\CC$ of $\RG_G$-modules
as the function $\HomDim \CC \colon  \Sub \to \bZ$ where for each $H\in \cF$, the integer 
$$(\HomDim \CC ) (H) = \hdim \CC(H)$$
is defined as the largest integer $d$ such that  $H_d (\CC (H))\neq 0$. If $H \notin \cF$, then we set $\un{n}(H) = -1$, as before.

Let us write $(H) \leq (K)$ whenever $g^{-1} Hg \leq K$ for some $g\in G$. Here $(H)$ denotes the set of subgroups conjugate to $H$ in $G$. The notation $(H) < (K)$ means that $(H)\leq (K)$ but $(H)\neq (K)$.

\begin{definition}\label{def:monotone}
We call a function $\un{n}\colon \Sub \to  \bZ$ \emph{monotone} if it
satisfies the property that $\un{n}(H) \geqslant \un{n}(K)$ whenever
$(H) \leq (K)$. We say that a monotone function  $\un{n}$ is
\emph{strictly monotone} if $\un{n}(H) > \un{n}(K)$, whenever $(H)< (K)$.  \qed 
\end{definition} 

We have the following:

\begin{lemma}\label{lem:monotone} The (chain) 
dimension function of every finite-dimensional projective chain complex $\CC$ of $\RG_G$-modules is monotone.
\end{lemma}

\begin{proof} Let $(L)\leq (K)$. If $\un{n}(K)=-1$, then the inequality $\un{n}(L) \geq \un{n}(K)$ is clear. So assume $\un{n}(K)=n \neq -1$. Then $\CC_n (K) \neq 0$. By the decomposition theorem for projective $\RG_G$-modules \cite[Chap.~I, Theorem 11.18]{tomDieck2}, every projective $\RG _G$-module  $P$ is of the form $P\cong \oplus _H E_H P_H$, where $H\in \cF$ and $P_H$ is a projective $N_G (H)/H$-module. Here the $R\G _G$-module $E_HP_H$ is defined by $$E_H P_H (?)=P_H \otimes _{RN_G(H)/H} RMap_G(G/?, G/H).$$ Applying this decomposition theorem to $\CC _n$, we observe that $\CC_n$ must have a summand $E_H P_H$ with $(K)\leq (H)$. But then $\CC_n(L)\neq 0$, and hence $\un{n}(L) \geq \un{n}(K)$.  
\end{proof}

We are particularly interested in chain complexes which have the homology of a sphere when evaluated 
at every $K \in \cF$. To specify the restriction maps in dimension zero, we will consider chain complexes $\CC$ which are equipped with  an augmentation map $\varepsilon\colon \CC _0 \to \un{R}$ such that $\varepsilon \circ \bd _1=0$. Here $\un{R}$ denotes the constant functor, and we assume that $\varepsilon(H)$ is surjective for $H \in \Supp(\CC)$.
We often consider $\varepsilon$ as a chain map $\CC \to \un{R}$ by considering $\un{R}$ as a chain complex over $\RG_G$ which is concentrated at zero. We denote a chain complex with an augmentation as a pair $(\CC, \varepsilon)$.

By the \emph{reduced homology} of a complex $(\CC,\varepsilon)$, we always mean the homology of the augmented chain complex 
$$ \widetilde \CC = \{ \cdots \to \CC_n \xrightarrow{\bd_n} \cdots  \to \CC _2 \xrightarrow{\bd_2} \CC _1\xrightarrow{\bd _1}
 \CC _0\xrightarrow{\varepsilon} \un{R} \to 0\}$$ 
where $\un{R}$ is considered to be at dimension $-1$. Note that the complex $\widetilde  \CC$ is the $-1$ shift of the mapping cone of the chain map $\varepsilon \colon \CC \to \un{R}$.    

\begin{definition}\label{def:Rhomologysphere}
 Let $\un{n}$ be a super class function supported on $\cF$, and let $\CC$ be a chain complex over $\RG_G$ with respect to a family $\cF$ of subgroups.  
 \begin{enumerate}
\item  We say that $\CC$ is an \emph{$R$-homology $\un {n}$-sphere} if there is an augmentation map $\varepsilon\colon \CC \to \un{R}$ such that the reduced homology of $\CC (K)$ is the same as the homology of an $\un{n}(K)$-sphere (with coefficients in $R$) for all $K \in \cF$. 
\item We say that $\CC$ is \emph{oriented} if the $W_G
(K)$-action  on the homology of $\CC (K)$ is trivial for all $K\in \cF$. 
\end{enumerate}
\end{definition}

Note
that we do not assume that the dimension function is strictly
monotone as in Definition II.10.1 in \cite{tomDieck2}. 

In transformation group theory, a $G$-CW-complex $X$ is called a \emph{homotopy representation} if $X^H$ is homotopy equivalent to the sphere $S^{\un{n}(H)}$ where $\un{n}(H)=\dim X^H$ for every $H\leq G$ (see tom Dieck  \cite[Section II.10]{tomDieck2}). We now introduce an algebraic analogue of this useful notion for chain complexes over the orbit category.  

In \cite[II.10]{tomDieck2}, there is a list of properties that are satisfied by homotopy representations. We will use algebraic versions of these properties to define an analogous notion for chain complexes.

\begin{definition}\label{def:algrep} Let $\CC$ be a finite 
projective  chain complex over $\RG_G$, which is an $R$-homology $\un{n}$-sphere.
We say $\CC$ is an \emph{algebraic homotopy representation} (over $R$) if 
\begin{enumerate}
\item The function $\un{n}$ is a monotone
function.
\item If $H,K \in \cF$ are such that $n=\un{n}(K)=\un{n}(H)$,
then for every $G$-map $f \colon  G/H \to G/K$ the induced map $\CC(f)\colon \CC (K) \to \CC (H)$ is an $R$-homology isomorphism.
\item Suppose $H, K, L\in \cF$ are such that $H \leq K,L$ and
let $M=\langle K, L \rangle$ be the subgroup of $G$ generated by $K$ and $L$. If
$n=\un{n}(H)=\un{n}(K)=\un{n}(L) >-1$, then $M \in \cF$ and
$n=\un{n}(M)$.  
\end{enumerate}
\end{definition}

Note that conditions (ii) and (iii) of Definition \ref{def:algrep} are automatic if the dimension function $\un{n}$ is strictly monotone. Under condition (iii),  the isotropy family $\cF$ has an important maximality property.

\begin{proposition}\label{cor:max} Let $\un{n}$ be a super class function and let $\CC$ be a %
projective chain complex of $\RG_G$-modules, which is an $R$-homology $\un{n}$-sphere. If condition \textup{(iii)} holds, then for each $H \in \cF$, the set of subgroups $\cF_H = \{ K \in \cF \vv (H) \leq (K),\ \un{n}(K) = \un{n}(H)> -1\}$ has a unique maximal element, up to conjugation. 
\end{proposition}

\begin{proof} Clear by induction from the statement of condition (iii). 
\end{proof}

In the remainder of this section we will assume that $R$ is a principal ideal domain. The important examples for us are $R =\bZp$ or $R=\bZ$. The main result of this section is the following proposition. 
 
\begin{proposition}
\label{prop:tightness conditions} Let $\un{n}$ be a super class function and $\CC$ be 
a finite projective  chain complex over $\RG_G$, which is an $R$-homology $\un{n}$-sphere. Assume that $R$ is a principal ideal domain. If the equality $\un{n} = \Dim \CC$ holds, then $\CC$ is an algebraic homotopy representation.
\end{proposition}

Before we prove Proposition \ref{prop:tightness conditions}, we make some observations and give some definitions for projective chain complexes.

\begin{lemma}\label{lem:resmap} Let $\CC$ be a projective chain complex over $\RG_G$. Then, for every $G$-map $f\colon  G/H \to G/K$, the induced map $\CC (f) \colon  \CC(K)\to \CC(H)$  is an injective map with an $R$-torsion free cokernel.
\end{lemma}

\begin{proof} It is enough to show that if $P$ a projective $\RG _G$-module, then for every $G$-map $f\colon  G/H \to G/K$, the induced map $P(f) \colon  P(K)\to P(H)$  is an injective map with a torsion free cokernel. Since every projective module is a direct summand of a free module, it is enough to prove this for a free module $P=R [X^?]$, where $X$ is a finite $G$-set. Let $f\colon G/H \to G/K$ be the $G$-map defined by $f(H)=gK$. Then the induced map $P(f)\colon R[X^K]\to R[X^H]$ is the linearization of the map $X^K \to X^H$ given by $x\mapsto gx$. Since this map is one-to-one, we can conclude that $P(f)$ is injective  with torsion free cokernel.  
\end{proof}

When $H\leq K$ and $f\colon G/H \to G/K$ is the $G$-map defined by $f(gH)=gK$ for each $g \in G$, then we denote the induced map  $\CC(f)\colon \CC (K) \to \CC (H)$ by $r^K_H$ and call it the \emph{restriction} map. When $H$ and $K$ are conjugate, so that $K=x^{-1}Hx$ for some $x \in G$, then the map $\CC(f)\colon\CC (K)\to \CC (H)$ induced by the $G$-map $f\colon G/H \to G/K$ defined by $f(gH)=gxK$ for each $g\in G$, is called the \emph{conjugation} map and usually denoted by $c^g_K$. Every $G$-map can be written as a composition of two $G$-maps of the above two types, so every induced map $\CC (f)\colon \CC (K)\to \CC (H)$ can be written as a composition of restriction and conjugation maps.

Since conjugation maps have inverses, they are always isomorphisms. So, the condition (ii) of Definition \ref{def:algrep} is actually a statement only about restriction maps. To study the restriction maps more closely, we consider the image of $r^K _H\colon \CC (K) \to \CC (H)$ for a pair $H \leq K$ and denote it by $\CC ^K_H$. Note that $\CC ^K_H$ is a subcomplex of $\CC (H)$ as a chain complex of $R$-modules. We also remark that  $\CC ^K_H$ is isomorphic to $\CC (K)$, as a chain complex of $R$-modules,  by Lemma \ref{lem:resmap},  whenever $\CC$ is a projective chain complex.

\begin{lemma}\label{lem:intersection} Let $\CC $ be a projective chain complex over $\RG _G$. Suppose that $K,L\in \cF$ are such that $H\leq K$ and $H\leq L$, and let $M=\la K, L\ra$ be the subgroup generated by $K$ and $L$. If\, $\CC ^K_H \cap \CC ^L _H\neq 0$ then $M \in \cF$,  and hence we have $\CC ^K_H \cap \CC ^L _H =\CC ^M _H$.   
\end{lemma}

\begin{proof} As before it is enough to prove this for a free $\RG_G$-module $P=R[X^?]$ where $X$ is a finite $G$-set whose isotropy subgroups lie in $\cF$. The restriction maps $r^K_H$ and $r^L_H$ are linearizations of the maps $X^K \to X^H$  and $X^L \to X^H$, respectively, which are  defined by inclusion of subsets. Then it is clear that the intersection of images of $r^K_H$ and $r^L_H$ (if non-zero) would be $R[X^K\cap X^L]$, considered as an $R$-submodule of $R[X^H]$. We have $X^K\cap X^L=X^M$ where $M=\la K, L \rangle$.  Therefore,  if $\CC _H ^K \cap \CC _H ^L\neq 0$, then we must have $X^M \neq \emptyset$ which implies that $M \in \cF$. Thus $\CC _H^M$ is defined and we can write $\CC ^K _H \cap \CC ^L _H= \CC ^M _H$ by the above fixed point formula.  
\end{proof}

Now, we are ready to prove Proposition \ref{prop:tightness conditions}.
 
\begin{proof}[Proof of Proposition \ref{prop:tightness conditions}]  The first condition in Definition \ref{def:algrep}  follows from Lemma
\ref{lem:monotone}. For (ii) and (iii), we use the arguments similar to the arguments given
in II.10.12 and II.10.13 in \cite{tomDieck2}. 

To prove (ii), let $f\colon G/H \to G/K$ be a $G$-map.  By Lemma \ref{lem:resmap}, the induced map $\CC (f): \CC (K) \to \CC (H)$ is injective with torsion free cokernel.  
Let $\DD$ denote the cokernel of $\CC (f)$. Then we have a short exact sequence of $R$-modules
 $$0\to \CC(K) \to \CC (H) \to \DD \to 0 $$ where both $\CC (K)$ and $\CC (H)$ have dimension $n$. Now consider the long exact \emph{reduced} homology sequence 
  (with coefficients in $R$) associated to this short exact sequence:
$$ \dots \to 0 \to H_{n+1} ( \DD ) \to H_n (\CC (K) )
\xrightarrow{f^*} H_n (\CC (H)) \to H_n (\DD ) \to \cdots $$
Note that $\DD$ has dimension less than or equal to $n$, so $H_{n+1}(\DD)=0$ and $H_n (\DD)$ is torsion free.
Since $H_n (\CC (K))=H_n (\CC (H))=R$, it follows that $f^*$ is an isomorphism. Since both $\CC(K)$ and $\CC(H)$ have no other reduced homology, we conclude that $\CC (f)$ induces an $R$-homology isomorphism between associated augmented complexes. Since the induced map $\un R (f): \un{R} (K)\to \un{R} (H)$ is the identity map $\id\colon  R\to R$, the chain map $\CC(f) \colon  \CC(K) \to \CC(H)$ is an $R$-homology isomorphism. 

To prove (iii), observe that there is a Mayer-Vietoris type exact sequence associated to the pair of complexes $\CC ^K_H$ and $\CC ^L_H$ which gives an exact sequence of the form
$$ 0 \to H_n (\CC^K_H \cap \CC ^L_H ) \to H_n (\CC ^K_H) \oplus H_n (\CC ^L_H )
 \to H_n (\CC ^K _H+ \CC ^L _H )  \to H_{n-1} (\CC ^K_H \cap \CC ^L_H  ) \to 0.$$ 
Here we again take the homology sequence as the reduced homology sequence. 

Let $i^K\colon \CC^K_H \to \CC(H)$, $i^L_H\colon \CC^L_H\to \CC(H)$, and  $j\colon \CC^K_H + \CC^L_H \to \CC(H)$ denote the inclusion maps.
We have zero on the left-most term since $\CC ^K_H +\CC ^L_H$ is an $n$-dimensional complex. To see the zero on the right-most term, note that by Lemma \ref{lem:resmap},  $\CC ^K _H \cong \CC (K)$ and $\CC ^L _H \cong \CC (L)$ as chain complexes of $R$-modules, so they have the same homology. This gives that $H_{i} (\CC ^K_H)=H_{i} (\CC ^L _H)=0$ for $i \leq n-1$. 

Also note that by part (ii), the composition
$$ H_n (\CC (K)) \cong H_n (\CC ^K _H)\xrightarrow{i^K_*} H_n (\CC ^K _H + \CC ^L _H ) \xrightarrow{j_*} H_n (\CC (H) )$$ is an isomorphism. So, $j_*$ is surjective. Since $H_{n+1} (\CC (H) /(\CC ^K_H +\CC ^L_H))=0$, we see that $j_*$ is also injective. Therefore, $j_*$
is an isomorphism. This implies that $i^K _*$ is an isomorphism. Similarly one can show that $i^L _*\colon H_n (\CC ^L _H )\to H_n (\CC ^K _H +\CC ^L _H)$ is also an isomorphism.
Using these isomorphisms and looking at the exact sequence above, we conclude that $H_n (\CC ^K _H \cap \CC ^L _H )\cong R$ and $H_i (\CC ^K _H \cap \CC ^L _H )=0$ for $i\leq n-1$.
So, $\CC ^K _H \cap \CC ^L _H$ is an $R$-homology $n$-sphere.

Since $n>-1$, this implies that
$\CC _H ^K \cap \CC _H ^L \neq 0$, and hence $M = \la K, L\ra \in \cF$ by Lemma \ref{lem:intersection}. Moreover, $\CC _H ^K \cap \CC _H ^L  = \CC^M_H$. This proves that $\un{n}(M)=n$ as desired. 
\end{proof}

\section{The Proof of Theorem A}\label{sec:thma}

In this section we will again assume that $R$ is a principal ideal domain. The main examples for us are $R =\bZp$ or $R=\bZ$, as before. 

\begin{definition} We say a  chain complex $\CC$ of $\RG_G$-modules 
is {\it tight at $H\in \cF$} if $$\dim \CC (H)=\hd \CC (H).$$  We call a chain complex of $\RG _G$-modules {\it tight} if it is tight at every $H\in \cF$. 
\end{definition}

Suppose that $\CC$ is a finite projective complex over $\RG_G$ which is an $R$-homology $\un{n}$-sphere.
If $\CC$ is chain homotopy equivalent to a tight complex, then Proposition \ref{prop:tightness conditions} shows that $\CC$ is an algebraic homotopy representation. This establishes one direction of Theorem A. The other direction uses the assumption that the chain modules of $\CC$ are free over $\RG_G$.
 
\begin{theorem}\label{thm:tightnessR}
Let $\CC$ be a finite 
chain complex of free $\RG_G$-modules which is a
homology $\un{n}$-sphere. If  $\CC$ is an algebraic homotopy representation over $R$, then $\CC$ is chain homotopy equivalent to a finite
free chain complex $\DD$ which is tight.\end{theorem}

\begin{remark} If $\CC$ is a finite projective chain complex,  then the analogous result holds for a sufficiently large $k$-fold join tensor product $\CC' =\bigast_k \CC$, by \cite[Theorem 7.6]{hpy1}. \end{remark}

We need to show that the complex $\CC$ can be made \emph{tight} at each $H \in \cF$ by replacing it with a chain complex homotopic to $\CC$. The proof is given in several steps.

\subsection{Tightness at maximal isotropy subgroups}
Let $H$ be a \emph{maximal element} in $\cF$. Consider the subcomplex
$\CC ^{(H)}$ of $\CC$ formed by free summands of $\CC$ isomorphic to
$\uR{G/H}$. Note that $\CC ^{(H)}$ is a submodule because $\Hom _{R\G_G} (R[G/H^?], R[G/K ^?])\neq 0$ only if $(H) \leq (K)$, and since $H$ is maximal, we have $\bd _i (\CC _i ^{(H)}) \subseteq \CC _{i-1} ^{(H)}$ for all $i$. The complex $\CC ^{(H)} $ is a complex of isotypic modules
of type $\uR{G/H}$. Recall that free $\RG_G$-module $F$ is called \emph{isotypic} of type $G/H$ if
it is isomorphic to a direct sum of copies of a free module
$\uR{G/H}$, for some $H \in \cF$. For extensions involving isotypic modules we have the following:

\begin{lemma}\label{lem:splitting}
Let $$\cE\colon   0 \to F \to F' \to M\to 0$$ be a short exact sequence of $\RG_G$-modules such that both $F$ and $F'$ are isotypic free
modules of the same type $G/H$. If $M(H)$ is $R$-torsion free, then
$\cE$ splits and $M$ is stably free.
\end{lemma}

\begin{proof} This is Lemma 8.6 of \cite{hpy1}. The assumption that $R$ is a principal ideal domain ensures that finitely generated $R$-torsion free modules are free.
\end{proof}

Note that  $\CC ^{(H)}(H)=\CC (H)$, since $H$ is maximal in $\cF$. 
This means that $\CC ^{(H)}$ is a finite free chain complex over $\RG_G$ of the form
$$\CC ^{(H)}\colon 0\to F _{d}
\to F_{d-1}\to \cdots \to F_1 \to F_0 \to 0$$ 
which  is a $R$-homology $\un{n} (H)$-sphere, with $\un{n} (H) \leq d $.

\begin{lemma}\label{lem:maximals} Let $\CC$ be a finite chain complex of free $\RG _G$-modules. Then $\CC$ is chain homotopy equivalent to a finite free chain complex $\DD$ which is tight at every maximal element $H \in \cF$.
\end{lemma}

\begin{proof} We apply \cite[Proposition 8.7]{hpy1} to the subcomplex $\CC^{(H)}$,  for each maximal element $H \in \cF$. The key step is provided by Lemma \ref{lem:splitting}.
\end{proof}

\subsection{The inductive step}
To make the complex $\CC$ tight at every $H\in \cF$ we use a downward induction, but the situation at an intermediate step
is more complicated than the first step considered above.

Suppose that $H\in \cF$ is such that $\CC$ is tight at every $K\in \cF$ such that $(K)>(H)$. Let $\CC
^{(H)}$ denote the subcomplex of $\CC$ with free summands of type
$\uR{G/K}$ satisfying $(H)\leq (K)$. In a similar way, we can define
the subcomplex $\CC ^{>(H)} $ of $\CC$ whose free summands are of type
$\uR{G/K}$ with $(H)< (K)$. The complex $\CC ^{>{(H)}}$ is a subcomplex
of $\CC ^{(H)}$. Let us denote the quotient complex $\CC ^{(H)} / \CC ^{>(H)}$
by $\CC _{(H)}$. The complex $\CC _{(H)}$ is isotypic with isotropy type $\uR{G/H}$. We have a short exact sequence of chain complexes
of free $\RG_G$-modules
$$ 0 \to \CC ^ {>(H)} \to \CC ^{(H)} \to \CC _{(H)} \to 0.$$
By evaluating at $H$, we obtain an exact sequence of chain complexes
$$0 \to \CC^{>(H)}(H) \to \CC^{(H)}(H) \to \CC_{(H)}(H) \to 0.$$ 
Since $\CC ^{(H)}(H)=\CC (H)$ and the image of the map on the left is generated by summands of the form $R[G/K^?]$ with $(H)<(K)$, the complex $\CC _{(H)}(H)$ is isomorphic to $S_H \CC $ as an $R[N_G(H)/H]$-module. Here $S_H$ denotes is splitting functor defined more generally for any module over an EI-category (see \cite[Definition 9.26]{lueck3}).

We also have an exact sequence
$$ 0 \to \CC ^{(H)} \to
\CC \to \CC / \CC ^{(H)} \to 0.$$
 If we can show that $\CC ^{(H)}$ is
homotopy equivalent to a complex $\DD '$  which is tight at $H$, then by taking the push-out of  $\DD'$ along the injective map $\CC^{(H)} \to \CC$, we can find a complex $\DD$ homotopy equivalent to 
$\CC$ which is tight at every $K\in \cF$ with $(K) \geq (H)$. So it is enough to show that $\CC ^{(H)} $ is homotopy equivalent to a complex $\DD '$ which is tight at $H$.

\begin{lemma}\label{lem:intermediate} Let $\CC$ be a finite free chain complex of $\RG_G$-modules,  such that $\CC$ is tight at every $K\in \cF$ with $(K)>(H)$, for some $H \in \cF$.
 Suppose 
\begin{enumerate}
\item $n=\hd \CC (H)\geq \dim \CC (K)$, for all $(K)>(H)$,  
and that 
\item $H_{n+1} (S_H\CC)=0$. 
\end{enumerate}
Then $\CC ^{(H)}$ is homotopy equivalent to a finite free chain complex $\DD '$ which is tight at every $K \in \cF$ with $(K)\geq(H)$.  
\end{lemma}

\begin{proof}  Let us fix $H \in \cF$ and assume that $\CC$ is tight at every $K\in \cF$ with $(K)>(H)$. We first observe that  $\CC ^{>(H)}$ has dimension $\leq n=\hdim \CC(H)$, since $\CC^{>(H)}(K) =\CC(K)$ for $(K)>(H)$, and $\dim \CC(K) \leq n$.  Let $d=\dim \CC (H)$. If $d=n$, then we are done, so assume that $d>n$. Then $\dim \CC_{(H)} = d$, and $\CC _{(H)}$ is a complex of the
form $$\CC _{(H)}\colon 0\to F _{d} \to F_{d-1}\to \cdots \to F_1 \to F_0
\to 0.$$ 

We claim that the map $\bd _d\colon  F_d \to F_{d-1}$ in the above chain complex is injective. 
Since $\CC_{(H)}$ is isotypic of type $(H)$, it is enough to show that this map is injective when it is calculated at $H$. To see this observe that the map $\bd_d$ is the same as the map obtained by applying the functor $E_H$ to the $N_G(H)/H$-homomorphism 
$\bd_d (H)\colon  F_d(H) \to F_{d-1}(H)$ (see \cite[Lemma 9.31]{lueck3}). Since the functor $E_H$ is exact, we have $\ker \bd_d= E_H (\ker \bd _d(H))$. Hence, if $\bd_d(H)$ is injective, then $\bd_n$ is injective.

We will show that $H_d (\CC _{(H)} (H) )=H_d(S_H\CC)=0$. 
 To see this consider the short exact sequence $0\to \CC ^{>(H)}(H) \to \CC (H)\to S_H\CC \to 0$. Since the complex $\CC ^{>(H)}$ has dimension $\leq n$, the corresponding long exact sequence gives that $H_d (S_H\CC ) \cong H_d (\CC (H))=0 $ when $d>n+1$. If $d=n+1$, then this is true by assumption (ii) in the lemma. 
Now we apply \cite[Proposition 8.7]{hpy1} to $\CC_{(H)}$ to obtain a tight complex $\DD'' \simeq \CC_{(H)}$, and then let $\DD'\simeq \CC^{(H)}$ denote the pullback of $\DD''$ along the surjection $\CC^{(H)} \to  \CC_{(H)}$. 
\end{proof}

\subsection{Verifying the hypothesis for the inductive step}

To complete the proof of Theorem \ref{thm:tightnessR}, we need to show that the assumptions in Lemma \ref{lem:intermediate} hold at an intermediate step of the downward induction. We will make detailed use of the internal homological conditions (i), (ii), and (iii) in Definition \ref{def:algrep}, satisfied by an algebraic homotopy representation $\CC$. We proceed as follows:

\begin{enumerate}
\item[(1)] The  dimension assumptions in Lemma \ref{lem:intermediate} follow from the condition (i), since when $\un{n}$ is monotone, we have 
$$n:=\hd \CC (H)=\un{n} (H)\geq \un{n} (K)=\hd \CC (K)=\dim\CC (K)$$ for all $K \in \cF$ with $(K)>(H)$. 

\item[(2)] The assumption that $H_{n+1} (S_H\CC)=0$ is established in Corollary \ref{cor:vanishing}. It  follows from the conditions (ii) and (iii) and the Mayer-Vietoris argument given below.
\end{enumerate}

In the rest of the section, we assume that $\CC$ is a finite projective chain complex of $\RG_G$-modules, which is an $R$-homology $\un{n}$-sphere, and satisfies 
the conditions \textup{(i)}, \textup{(ii)}, and \textup{(iii)} in Definition \ref{def:algrep}.  Assume also that $\CC$ is tight for all $K\in \cF$ with $(K)>(H)$ for some fixed subgroup $H \in \cF$.  We will say $\CC$ is  
 \emph{tight above $H$}, for short.
  Let $\cK_H$ denote the set of all subgroups
$$\cK_H = \{ K \in \cF \vv K >H \text{\ and\ } n:= \un{n} (K)=\un{n}(H)\}.$$

Let $\CC$ be an algebraic homotopy representation, which is tight above $H$.
Let $\CC^{K} _H$ denote the image of the restriction map 
$$r^{K} _H\colon \CC (K)\to \CC (H),$$ 
for every $K \in \cF$ with $K \geq H$. 
Then $\CC _H ^{K} $ is a subcomplex of $\CC (H)$ and by Lemma \ref{lem:resmap}, it is isomorphic to $\CC (K)$. By condition (iii) of Definition \ref{def:algrep}, the collection $\cK_H$ has a unique maximal element $M$. 
In addition, we  have the equality
$$\CC^{>(H)}(H) = \sum _{K \in \cK_H} \CC ^{K} _H,$$
since $(G/K)^H$ is the union of the subspaces $ (G/K)^L $, with $L > H$ and  $(L) = (K)$. 

Moreover, if $K \in \cK_H$, then by condition (ii), the subcomplex $\CC ^{K} _H$ is an $R$-homology $n$-sphere and the map $$H_n (\CC ^M _H )\to  H_n (\CC ^{K} _H)$$
induced by the inclusion map $\CC_H ^M \hookrightarrow \CC^K_H$ is an isomorphism. More generally, the following also holds.

\begin{lemma}\label{lem:tightness}
Let $\CC$ be an algebraic homotopy representation which is tight above $H$,  for some fixed $H\in \cF$, 
and let $K_1,\dots , K_m$ be a set of subgroups in $\cK_H$.  Then the subcomplex $\sum _{i=1} ^m \CC ^{K_i} _H $ is an $R$-homology $n$-sphere and the map
\eqncount
\begin{equation}\label{statement}
H_n (\CC ^M _H )\to  H_n (\sum _{i=1} ^{m} \CC ^{K_i} _H)\end{equation}
 induced by the inclusion maps  is an isomorphism.
\end{lemma}

\begin{proof} This follows from the Mayer-Vietoris spectral sequence in algebraic topology
(see \cite[pp.~166-168]{brown1}), which computes the homology of a union of spaces $X=\bigcup  X_i$ in terms of the homology of the subspaces and their intersections. We include a direct argument for the reader's convenience. 
  
The case $m=1$ follows from the remarks above. For $m>1$, we have the following Mayer-Vietoris type long
exact sequence
$$0 \to H_n (\DD _{m-1} \cap \CC^{K_m}_H) \to H_n (\DD_{m-1}) \oplus H_n (\CC
^{K_m} _H)  \to H_n (\DD_m ) \to H_{n-1} (\DD_{m-1} \cap \CC ^{K_m} _H )\to $$ where $\DD_j= \sum _{i=1} ^j
\CC ^{K_i } _H$ for $j=m-1,m$. By the inductive assumption, we know that $\DD _{m-1}$ is an $R$-homology $n$-sphere and the map $H_n (\CC ^M _H ) \to H_n (\DD_{m-1})$ induced by inclusion 
is an isomorphism.  

We have
 $$\DD _{m-1} \cap \CC _H ^{K_m}=(\sum _{i=1} ^{m-1} \CC ^{K_i} _H )\cap \CC ^{K_m} _H =\sum _{i=1} ^{m-1} (\CC ^{K_i} _H \cap \CC ^{K_m} _H )= \sum _{i=1} ^{m-1} \CC _H ^{\la K_i , K_m \ra}$$ where the last equality follows from  Lemma \ref{lem:intersection}. We can apply Lemma \ref{lem:intersection} here because $\CC_H^M \subseteq \CC_H ^{K}$ for all $K \in \cK_H$ gives that $\CC _H ^{K_i}\cap \CC_H ^{K_m}\neq 0$ for every $i=1,\dots, m-1$. We also obtain $\langle K_i , K_m\rangle \in \cK_H$ for all $i$. Applying our inductive assumption again to these subgroups, we conclude that $\DD _{m-1}\cap \CC ^{K_m} _H $ is an $R$-homology $n$-sphere and that the map 
$$H_n (\CC ^M _H )\to H_n (\DD _{m-1} \cap \CC ^{K_m} _H )$$
induced by inclusion 
 is an isomorphism. This gives that $H_i (\DD _m )=0$ for $i\leq n-1$. We also obtain a commuting diagram 
$$\xymatrix{0 \ar[r] & H_n (\CC ^M _H )  \ar[r] \ar[d]& H_n (\CC ^M _H )\oplus H_n (\CC ^M _H )\ar[r]\ar[d] & H_n (\CC ^M _H ) \ar[r]\ar[d]^{\varphi} & 0 \\
0 \ar[r] & H_n (\DD _{m-1} \cap \CC^{K_m}_H) \ar[r] & H_n (\DD_{m-1}) \oplus H_n (\CC
^{K_m} _H)  \ar[r] & H_n (\DD_m ) \ar[r] & 0} $$
Since all the vertical maps except the map $\varphi$ are known to be isomorphisms, we obtain that $\varphi$ is also an isomorphism
 by the five lemma. This completes the proof.
\end{proof}

\begin{corollary}\label{cor:vanishing} Let $\CC$ be an algebraic homotopy representation which is tight above $H$, for some fixed $H\in \cF$. 
Then $H_{n+1}(S_H\CC) = 0$.
\end{corollary}

\begin{proof} Let $\cK_H=\{K_1, \dots , K_m\}$. By condition (ii), we know that the composition 
$$H_n (\CC (M) ) \xrightarrow{\cong} H_n (\CC ^M _H )\to H_n (\sum _{i=1} ^{m} \CC ^{K_i} _H)\to H_n (\CC (H))$$
is an isomorphism. However, we have just proved that the middle map is an isomorphism, and that all the modules involved in the composition are isomorphic to $R$. Therefore, the map induced by inclusion
$$H_n (\sum _{i=1} ^{m} \CC ^{K_i} _H)\to H_n (\CC (H))$$
is an isomorphism. 
Since $\CC$ is tight above $H$,  we have $\dim \CC (K)<n$ whenever $(H)\leq (K)$ and $\un{n} (K) < n$, for some $K\in \cF$.  This implies the relation
$$H_n(\CC^{>(H)}(H))= H_n (\sum _{i=1} ^{m} \CC ^{K_i} _H) \cong H_n(\CC(H))$$ where the isomorphism is induced by the the inclusion of chain complexes. 
From the exact sequence $0 \to \CC^{>(H)}(H) \to \CC(H) \to S_H\CC \to 0$, 
 and the fact that $\hd\CC (H) =n$, we conclude that
$H_{n+1}(S_H\CC)=0$, as required. 
\end{proof}

This completes the proof of Theorem \ref{thm:tightnessR} and hence the proof of Theorem A. In \cite{hpy1}, we proved the following realization theorem for free $\ZG_G$-module chain complexes,  with respect to any family $\cF$, which are $\bZ$-homology $\un{n}$-spheres satisfying certain extra conditions.

\begin{theorem}[{\cite[Theorem 8.10]{hpy1}, \cite{pamuks1}}]\label{thm:realization} Let $\CC$ be a finite  
chain complex of free $\ZG_G$-modules which is a $\bZ$-homology $\un{n}$-sphere.
Suppose that $\un{n} (K) \geq 3$ for all $K \in \cF$. If $\CC _i
(H)=0$ for all $i > \un{n} (H) +1$, and all $H\in \cF$, then there
is a finite $G$-CW-complex $X$ with isotropy in $\cF$, 
such that $\uCZ{X}$ is chain homotopy
equivalent to $\CC$ as chain complexes of $\ZG_G$-modules.
\end{theorem}

Note that a $\bZ$-homology $\un{n}$-sphere $\CC$  with $\Dim \CC = \un{n}$, and  $\un{n} (K) \geq 3$ for all $K\in \cF$, will automatically satisfy these conditions. So Corollary B follows immediately from
Theorem A and Theorem \ref{thm:realization}.

\begin{remark}\label{simplyconnected} The construction actually produces a finite $G$-CW-complex $X$ with the additional property that all the non-empty fixed sets $X^H$ are simply-connected. Moreover, by construction, $W_G(H) = N_G(H)/H$ will act trivially on the homology of $X^H$. Therefore $X$ will be an  \emph{oriented} geometric homotopy representation (in the sense of tom Dieck). From the perspective of Theorem A, since we don't specify any dimension function, a $G$-CW-complex $X$ with all fixed sets $X^H$ integral homology spheres will lead (by three-fold join) to a homotopy representation. The same necessary and sufficient conditions for existence apply.  
\end{remark}

\section{Inflation and deflation of chain complexes}\label{sec:deflation}
In this section we define two general operations on chain complexes in preparation for the proof of Theorem C.  For a finite $G$-CW complex $X$ which is a mod-$p$ homology sphere, the Borel-Smith conditions can be
proved using a reduction argument to certain $p$-group subquotients (compare \cite[III.4]{tomDieck2}). 
For a subquotient $K/L$, the reduction comes from considering the
fixed point space $X^L$ as a $K$-space. To do a similar reduction for chain complexes over $\RG _G$, we first introduce a new functor for $\RG _G$-modules, called the \emph{deflation} functor. We will introduce this functor as a restriction functor between corresponding module categories. For this discussion $R$ can be taken as any commutative ring with $1$ and $\cF_G$ is any family subject to the extra conditions we assume during the construction.

Let $N$ be a normal subgroup of $G$. We define a functor $$F \colon
\G _{G/N} \to \G _G $$ by considering a $G/N$-set (or $G/N$-map) as
a $G$-set (or $G$-map) via composition with the quotient map $G \to G/N$. For this definition to make sense, the families $\cF
_{G/N}$ and $\cF _G$ should satisfy the property that if $K\geq N$ is such that $(K/N) \in
\cF _{G/N} $, then $K \in \cF _G$. Since we always
assume the families are nonempty, the above assumption also implies
that $N \in \cF_G$. For notational simplicity from now on, let us denote $K/N$ by $\overline K$ for every $K \geq N$.

If a family $\cF_G$ is already given, we will always take $\cF
_{G/N} = \{ \overline  K \vv  K \geq N \text{ and } K \in \cF_G \}$ and the condition
above will be automatically satisfied. We also assume that $N \in \cF _G$ to have a nonempty family for
$\cF _{G/N}$.

The functor $F$ gives rise to two functors  (see \cite[9.15]{lueck3}): 
$$\res _F \colon  \Mod\text{-}\RG_{G} \to \Mod\text{-}\RG_{G/N}$$ and
$$\ind _F \colon  \Mod\text{-}\RG_{G/N} \to \Mod\text{-}\RG_{G}\ .$$
The first functor $\res _F$ takes a $\RG _G$-module $M$ to the $\RG
_{G/N}$-module $$\Def ^G _{G/N} (M):= M \circ F \colon  \G _{G/N} \to
\RMod.$$ We call this functor the {\em deflation functor}. Note that
$$(\Def ^G _{G/N} M ) (\overline K)=M(K).$$
The induction functor $ \Inf _{G/N} ^G:= \ind _F$ associated to $F$
is called the {\em inflation functor}. 
For every $H \in \cF_G$, we have $$\Inf ^G _{G/N} (M) (H)= \Bigl 
(\bigoplus\nolimits _{\overline{K} \in \cF _{G/N}} M(\overline{K})\otimes _{R W_{\overline G} (\overline{K})} R\Map_G (G/H, G/K) \Bigr ) / \sim$$
where the relations come from the tensor product over $R\G_{G/N}$ (see \cite[Definition 9.12]{lueck3}). In general, it can be difficult to calculate $\Inf ^G _{G/N} M$ for an arbitrary $R\G _{G/N}$-module $M$. In the case where $M$ is a free $R\G_{G/N}$-module we have the following lemma.

\begin{lemma}\label{lem:inflation} Let $X$ be a finite $G/N$-set. Then,
we have $$\Inf ^G _{G/N} \uR{X}= \uRb{(\Inf ^G _{G/N} X)}.$$ 
\end{lemma}

\begin{proof} It is enough to show this when $X=\oG/\oK$ for some $K\leq G$ such that $K \geq N$.  In this case, $\uRb{(\oG/\oK)}$ is isomorphic to $E_{\oK}P_{\oK}$ where $P_{\oK}= R[W_{\oG}(\oK)]$. Since $E_{\oK}(-)$ is defined as induction $\Ind _{F'}(-)$ for the functor $F': R[W_{\oG} (\oK)]\to R\G_{G/N}$ (see \cite[9.30]{lueck3}), we have 
$$ \Inf ^G _{G/N} \uRb{(\oG/\oK)}=\Inf ^G _{G/N} E_{\oK} P_{\oK}= \Ind _{F} \Ind _{F'} P_{\oK}=\Ind _{F\circ F'} P_{\oK} $$
where $F: \G_{G/N} \to \G _G $ is the functor defined above. Since $W_{\oG}(\oK)\cong W_G (K)$, after suitable identification, the composition $F\circ F'$ becomes the same as the inclusion functor $i: W_G (K) \to \G_G$, so we have $$\ind _{F \circ F'}P_{\oK}=E_{K} RW_G (K) =\uRb{G/K}$$ as desired.
\end{proof}

By general properties of restriction and induction
functors associated to a functor $F$, the functor $\Def ^G _{G/N}$
is exact and $\Inf ^G _{G/N}$ respects projectives  (see \cite[9.24]{lueck3}). The deflation functor has the following formula for free modules.

\begin{lemma}\label{lem:deflation} Let $X$ be a $G$-set. Then,
we have $$\Def ^G _{G/N} \uR{X}= \uRb{(X^N)}$$
 In particular, if  $H
\in \cF_G$ implies $HN \in \cF_G$, then the functor $\Def ^G _{G/N}$
respects projectives.
\end{lemma}

\begin{proof} For every $K\in \cF _G$  such that $K\geq N$, we have
\begin{equation*}
\begin{split}
(\Def ^G _{G/N} \uR{X})(\overline K)&=\uR{X} (K)=R[X^K] 
=R[(X^N)^{K/N}]=\uRb{(X^N)} (\overline K). \\
\end{split}
\end{equation*}
Note that $(G/H)^N= G/HN$ as a 
$G/N$-set. If $H \in \cF_G$
implies $HN \in \cF_G$, then by assumption  $\overline{HN} \in \cF_{G/N}$.
Hence $\uRb{{((G/H)^N)}}$ is free as an $R\G_{G/N}$-module and $\Def ^G_{G/N}$ respects
projectives.
\end{proof}

\section{The Borel-Smith conditions for chain complexes}
\label{sect:Borel-Smith}

Let $G$ be a finite group, and let $X$ be a finite $G$-CW-complex which is a mod-$p$ homology sphere for some prime $p$. Then by Smith theory, the fixed point space $X^H$ is also a mod-$p$ homology sphere (or empty), for every $p$-subgroup $H\leq G$. So if we take $R=\bZ/p$ and $\G_G$ as the orbit category over the family $\cF_p$ of all $p$-subgroups of $G$, then the chain complex $\uCZ{X}$ over $R\G_G$ is a finite free chain complex which is an $R$-homology $\un{n}$-sphere. Here, as before, we take $\un{n}(H)=-1$ when $X^H=\emptyset$. In this case, it is known that
the super class function $\un{n}$ satisfies certain conditions
called the Borel-Smith conditions (see  \cite[Thm. 2.3 in Chapter XIII]{borel-seminar} or \cite[III.5]{tomDieck2}). These conditions are given as
follows:

\begin{definition}\label{def:Borel-Smith} Let $G$ be a finite group and 
let $f \colon  \Sub  \to \bZ$ be super class function, where $\Sub$ denotes the family of all subgroups of $G$.  We say the function $f$ satisfies the Borel-Smith conditions at a prime $p$,  if it has the following properties: 

\begin{enumerate}\addtolength{\itemsep}{0.2\baselineskip}
\item If $L \snor K\leq G$ are such that $K/L \cong \cy{p}$, and $p$ is odd,
then $f(L)-f(K)$ is even.

\item If $L\snor K\leq G$ are such that $K/L \cong \cy{p} \times \cy{p}$, and if
$L_i/L$ denote the subgroups of order $p$ in $K/L$, then $$f(L)-f(K)=\sum
_{i=0} ^p (f(L_i)-f(K)).$$  
\item If $p=2$, and $L\snor K \snor N \leq G$ are such that $L \snor N$, $K/L \cong \cy{2}$, and $N/L \cong \cy 4$,  then
$f(L)-f(K)$ is even.
\item  If $p=2$, and $L\snor K \snor N \leq G$ are such that $L \snor N$,  $K/L \cong \cy{2}$, and $N/L=Q_8$ is
 the quaternion group of order $8$, then  $f(L)-f(K)$ is
divisible by $4$.
\end{enumerate}
\end{definition}

We will show that these conditions are satisfied by the homological dimension
function $\un{n}$ of a finite projective complex $\CC$ over $\RG _G$ which is an $R$-homology $\un{n}$-sphere.  
Recall that $\un{n}(H) = -1$ whenever $H \notin \cF$, by Definition \ref{def:Rhomologysphere}.

\begin{thmc}
Let $G$ be a finite group, 
$R=\bZ/p$, and let $\cF$ be a given family of subgroups of $G$. If $\CC$ is a finite projective chain complex over $\RG _G$, which is an $R$-homology
$\un{n}$-sphere, then the function $\un{n}$ satisfies the Borel-Smith conditions  at the prime $p$.
\end{thmc}

The rest of the section is devoted to the proof of 
Theorem C. 
As a first step of the proof we extend the given family $\cF$ to the family $\Sub$ of all subgroups of $G$ by taking $\CC (H)=0$ for every $H \not \in \cF$. Over the extended family, $\CC$ is still a finite projective chain complex over $R\Gamma_G$ and an $R$-homology $\un{n}$-sphere.

The Borel-Smith conditions are conditions on subquotients $K/L$ where $L\snor K\leq G$. To show that a Borel-Smith condition holds for a particular subquotient group $K/L$, we  consider the complex $\Def ^K _{K/L} \res ^G _K \CC $ (see Section \ref{sec:deflation}). This is a finite projective complex over 
$R\G_{K/L}$ because both restriction and deflation functors preserve projectives (the condition in Lemma \ref{lem:deflation} is satisfied because we extended our family $\cF$ to the family of all subgroups of $G$). 

Our first observation is the following:

\begin{lemma}\label{lem:monotone2}
Let $G$ be a finite group and let
$R=\bZ/p$.
 If $\CC$ is a finite projective chain complex over $\RG _G$, which is an $R$-homology
$\un{n}$-sphere, then whenever $L\snor K \leq G$ and $K/L$ is a $p$-group, we have $\un{n}(L)\geq \un{n}(K)$.
\end{lemma}

\begin{proof} By the discussion above, it is enough to show that if $G=\cy p$ and $\CC$ is 
a finite projective $\RG _G$-complex which is an
$R$-homology $\un{n}$-sphere, then the inequality $\un{n}(1)\geq \un{n}(G)$ holds. Assume that $\un{n}(1)\neq \un{n}(G)$. Since $\un{R}$ is projective,  we can add $\CC _{-1} =\un{R}$ and consider the homology of the augmented complex $\widetilde \CC$. The complex $\widetilde \CC$ has nontrivial homology only at two dimensions, say $m$ and $k$ with $m> k$, so we get an extension of the form
$$ 0 \to H_m(\widetilde \CC) \to \widetilde\CC_{m}/ \Image \bd _{m+1}  \to \cdots \to
\widetilde \CC_{k+1} \to \ker \bd _{k} \to H_k (\widetilde \CC) \to 0.
$$
where the homology modules are $I_1 R$ and $I_G R$ in some order. 

For $H\in \cF$, the module $I_H M$ denotes the atomic module concentrated at $H$ with the value $(I_H M)(H)=M$ (see \cite[9.29]{lueck3}). We claim that  $H_m (\widetilde\CC)=I_1 R$ and $H_k (\widetilde\CC )=I_GR$, meaning that the module $I_GR$ appears before $I_1R$ in the homology. 
Once we show this,  it will imply that $\un{n}(1) > \un{n}(G)$ as desired.

 Let $\DD$ denote the chain complex obtained by erasing the homology groups $H_m(\widetilde \CC)$ and $H_k (\widetilde \CC)$ from the above exact sequence. Since $\ker \bd_k$ is projective and $\Image \bd _{m+1}$  has a finite projective resolution, the Ext-group $\Ext _{\RG_G} ^* (\DD, M)$ is zero after some fixed dimension, for every $R\G_G$-module $M$. 
 We will take $M = I_1R$ for simplicity. 
 
There is a two-line spectral sequence $E_2 ^{s,t}=\Ext_{R\G_G} ^s (H_t (\DD ), M )$ which converges to $\Ext ^* _{R\G_G} (\DD , M)$. 
Suppose, if possible, that $H_k (\widetilde \CC )=I_1R$. 
The module $I_1 R$ is concentrated at $1$, so its projective resolution is of the form $E_1 P_*$ for some projective resolution $P_*$ of $R$ as an $RG$-module. 
Then the 
bottom line of this spectral sequence $E_2 ^{*, 0} $ would be isomorphic to 
$$\Ext ^* _{R\G_G} (H_k (\widetilde \CC ) , M )=\Ext ^* _{R\G _G } (I_1 R, I_1R) = H^i (\Hom _{R\G_G} (E_1 P_*, I_1R))=H^*(G;R).$$
Since this cohomology ring is not finitely generated, there must be a non-trivial differential from the top line
$$\Ext ^* _{R\G_G} (H_m (\widetilde \CC ) , M )=\Ext ^* _{R\G _G } (I_G R, I_1R)$$
in order for the spectral sequence to converge to a finite dimensional limit.

 The differential of this spectral sequence is given by multiplication with an extension class in $\Ext_{R\G_G } ^{m-k+1} (I_1 R, I_G R)$.
  But, by a similar calculation as above, we see that 
 $$\Ext ^i _{\RG_G} (I_1R, I_GR)=H^i (G , (I_G R) (1)))=0$$ for all $i\geq 0$, because $(I_G R)(1)=0$. This contradiction shows that $H_k (\widetilde \CC )=I_GR$ and $H_m (\widetilde \CC )=I_1R$, as required.
 \end{proof}

The above lemma shows that under the conditions of Theorem C, 
 the dimension function $\un{n}$ is monotone in the sense defined in \cite[p.~211]{tomDieck2}. Now we verify (in separate steps) that the dimension function satisfies the conditions of Definition \ref{def:Borel-Smith}.
These conditions come from the period of
the cohomology of the corresponding subquotient groups. 

\begin{lemma}[Borel-Smith, part (i)]\label{lem:Case (i)} Let $G= \cy p$, for $p$ an odd prime, let $R=\cy p$, and $\CC$ be a finite projective $\RG _G$-complex which is an $R$-homology $\un{n}$-sphere.
Then $\un{n}(1)-\un{n}(G)$ is even. 
\end{lemma}
\begin{proof} 
 Consider
 the subcomplex $\widetilde \CC ^{(G)} $
of $\widetilde \CC$ consisting of all projectives of type $\uR{G/G}$, and
 let $\DD=\widetilde \CC/ \widetilde \CC^{(G)}$ denote the quotient complex. The complex $\DD$ has nontrivial
homology only in dimensions  $m$ and $k+1$, where $m=\un{n}(1)$ and $k=\un{n}(G)$.  
Moreover, all the $R\G_G$-modules in the complex $\DD$ are of the form $\uR{G/1}$. Evaluating at the subgroup $1$, we obtain a chain complex of free $RG$-modules 
$$ 0 \to Q_d \to \dots \to Q_{m+1} \maprt{\bd _{m+1}} Q_{m} \to \dots \to Q_{k+1}\maprt{\bd_{k+1}} Q_{k} \to \dots \to Q_0 \to 0.$$
whose homology is $R$ at dimensions $m$ and $k+1$.
 This gives an exact sequence of the form
$$ 0 \to R \to Q_{m}/ \Image \bd _{m+1}  \to \cdots \to
Q_{k+2} \to \ker \bd _{k+1} \to R \to 0.
$$ Using the fact that free $RG$-modules are both projective and injective, we conclude that all the modules in the above sequence, except the two $R$'s on the both ends, are projective as $RG$-modules, so we have a periodic resolution. Since the group $G=\cy p$ has periodic 
$R$-cohomology with period $2$, we have $m-k = \un{n}(1) - \un{n}(G) \equiv 0 \pmod 2$. 
\end{proof}
\begin{remark}
The $R$-cohomology of the group $G=\cy 2$ is periodic of period 1. 
\end{remark}

For condition (ii), the argument is more involved. As before, after the subquotient reduction we may assume that $G=K/L = \cy p \times \cy p$, and that $\cF$ is the family of all  subgroups of $G$.
Since the complex $\CC$ is a finite complex of projective modules,
for any $\RG_G$-module $M$, we have $$ H^n (\Hom _{\RG_G} (\CC , M )
)=0$$ for $n > d$, where $d$ is the dimension of the chain complex
$\CC$. Consider the
hyper-cohomology spectral sequence for the complex $\CC$. This is a
spectral-sequence with $E_2$-term given by
\eqncount
\begin{equation}\label{eqn:hyperSS}
E_2 ^{s,t} =\Ext ^s _{\RG _G} (H_t (\CC ), M) 
\end{equation}
which converges to $H^{s+t} (\Hom _{\RG _G} (\CC , M ))$. Since
$\un{R}$ is a projective $\RG_G$-module, we can replace $H_t (\CC)$ with the
reduced homology $\widetilde H_t (\CC)$. So, we have nonzero terms
for $E_2 ^{s,t}$ only when $t$ is equal to $n_1=\un{n}(1)$, $n_G=\un{n}(G)$, or $n_{H_i}=\un{n}(H_i)$ 
where $H_i$ are the subgroups of $G$ of order $p$. 
Since $\un{n}$ is monotone, we have $n_1 \geq n_{H_i} \geq n_G$ for all $i \in \{0,\dots, p\}$. The required formula is
$$n_1 - n_G = \sum_{i = 0}^{p} (n_{H_i} - n_G).$$

\begin{remark}
In the proof below we assume $n_1>n_{H_i}>n_G$ for all $i$, to make the argument easy to follow. If for some $i$, we have $n_{H_i}=n_1$ or $n_{H_i}=n_G$, then the argument below can be adjusted easily to include these cases as well. 
\end{remark} 

By adding free summands to the complex $\CC$, we can assume that
all the cohomology between dimensions $n_1$ and $n_G$ is concentrated at the dimension $n_M=\max _i \{ n_{H_i}\}$. Then the homology at this dimension will be an $\RG_G$-module 
which is filtered by Heller shifts of homology groups $H_t(\CC)$ at dimensions
$t=n_{H_i}$ for $i=0,\dots, p$. The homology of the complex $\CC$ at
dimension $n_{H_i}$ is $I_{H_i}R$, where $I_{H_i}R$ denotes the $R\G_G$ module with value $R$ at $H_i$ and zero at all the other subgroups. We have the following lemma. 

\begin{lemma}\label{lem:Ext-calc}
If $i, j \in \{0, \dots, p\}$ are such that $i\neq j$, then $$\Ext ^m _{\RG _G} ( I_{H_i}R , I_{H_j} R)=0 $$ for every $m \geq 0$. 
\end{lemma}

\begin{proof} The projective resolution of $I_{H_i}$ is formed by projective modules of type $E_H P$ with $H=1$ or $H_i$. Since $$\Hom _{\RG _G} (E_H P, I_{H_j} R)  \cong \Hom _{RW_G (H) } (P, I_{H_j} (H) )=0$$ when $i\neq j$, we obtain the desired result.
\end{proof} 

As a consequence of Lemma \ref{lem:Ext-calc}, we conclude that all the
extensions in this filtration of $H_{n_M} (\CC)$  are split extensions. So, the homology module $H_{n_M} (\CC)$ is isomorphic to a direct sum of Heller shifts of modules $I_{H_i}R$. In particular, we obtain that, for any $\RG_G$-module $M$,
$$ \Ext _{\RG_G } ^{s} ( H_{n_M} (\CC ) , M) \cong \oplus_i \Ext _{\RG_G}
^{s+n_M-n_{H_i}} (I_{H_i} R , M) $$
for every $s \geq 0$.

The spectral sequence given in (\ref{eqn:hyperSS}) converges to zero for total dimension $>d$. It has only three non-zero horizontal lines, so it gives a long exact sequence of the form
$$\cdots \to \Ext _{\RG _G} ^{k +n_1-n_G+1} (I_G R, M) \maprt{\delta} \Ext^k
_{\RG_G} (I_1 R, M) \maprt{\gamma} \oplus_{i=0}^{p} \Ext _{\RG _G}
^{k+n_1-n_{H_i}+1} (I_{H_i} R , M) $$ $$\to \Ext _{\RG _G} ^{k
+n_1-n_G+2} (I_G R, M) \maprt{\delta} \Ext _{\RG _G} ^{k+1} (I_1R, M)
\to \cdots$$ where $k$ is an integer such that $k>d-n_1$ and $M$ is any $R\G_G$-module. If we take $M=I_1 R$, then  $\Ext^k _{\RG_G} (I_1 R, M)\cong H^k (G,
R)$. When $M=I_1 R$, the other Ext-groups in the above exact sequence also reduce to the cohomology of
the group $G$ with some dimension shifts. 

\begin{lemma}\label{lem:Ext-calc for S_i}
For every $i\in \{ 0, \dots, p\}$, we have
$$\Ext _{\RG _G}^{m} (I_{H_i} R , I_1R ) \cong \Ext _{\RG _G}^{m-1} (I_1 R , I_1R
)\cong H^{m-1} (G ; R)$$ for every $m\geq 1$. We also have
$$\Ext _{\RG _G}^{m} (I_{G} R , I_1R ) \cong  \oplus _p \Ext _{\RG _G}^{m-2} (I_1R , I_1R )\cong \oplus_p H^{m-2} (G;
R)$$ for every $m \geq 2$. Here $\oplus_p$ denotes the direct sum of $p$-copies of the same $R$-module.   
\end{lemma}

\begin{proof} Since we already observed that $\Ext^k _{\RG_G} (I_1 R, I_1R )\cong H^k (G,
R)$ for every $k \geq 0$, it is enough to show the first isomorphisms. Let $i\in \{0, \dots, p\}$ and $J_{H_i}R $
denote the $\RG _G$ module with value $R$ at subgroups $1$ and $H_i$ and zero at every other subgroup. We assume that the restriction map is an isomorphism. So we have a non-split exact sequence of $\RG _G$-modules of the form $$ 0 \to I_1 R \to J_{H_i} R \to I_{H_i } R \to 0.$$ 
Since the projective resolution of $J_{H_i} R$ will only include projective modules of the form $E_{H_i } P$, we have $\Ext ^m _{\RG _G } (J_{H_i} R , I_1 R ) =0$ for all $m\geq 0$. The long exact Ext-group sequence associated to the above short exact sequence will give the desired isomorphism for the module $I_{H_i} R$.

For the second statement in the lemma, we again only need to show that the isomorphism $$\Ext _{\RG _G}^{m} (I_{G} R , I_1R ) \cong  \oplus _p \Ext _{\RG _G}^{m-2} (I_1R , I_1R )$$ holds for all $m\geq 2$.
Let $N$ denote the $\RG _G$-module defined as the kernel of the map
$\un{R} \to I_G R$ which induces the identity homomorphism at $G$. Since the constant module $\un{R}$ is projective as a $\RG _G$-module, we have  
$$\Ext _{\RG}^{m} (I_{G} R , I_1R ) \cong \Ext _{\RG}^{m-1} (N , I_1R
)$$
for $m\geq 2$. In addition, there is an exact sequence of the form $$ 0 \to \oplus _{p} I_1 R \to 
\oplus _{i=0} ^{p} J_{H_i} R \to N \to 0.$$ Since 
$\Ext ^m _{\RG _G } (J_{H_i} R , I_1 R ) =0$ for all $m\geq 0$, we obtain
$$\Ext _{\RG}^{m} (I_{G} R , I_1R ) \cong \Ext _{\RG}^{m-1} (N , I_1R
)\cong \oplus _p \Ext _{\RG}^{m-2} (I_1R , I_1R )\cong \oplus_p H^{m-2} (G;
R)$$ for every $m\geq 2$. This completes the proof of the lemma.  
\end{proof}

\begin{lemma}[Borel-Smith, part (ii)]\label{lem:Case (ii)} Let $G = \cy p \times \cy p$, let $R = \cy p$, and let   $\CC$ be a finite projective $\RG _G$-complex which is an $R$-homology $\un{n}$-sphere.
Then 
$$ \un{n}(1) - \un{n}(G) = \sum_{i=0}^{p} (\un{n}(H_i) - \un{n}(G))$$
where $H_0, H_1, \dots, H_p$ denote the distinct subgroups of $G$ of order $p$.
\end{lemma}

\begin{proof}
Using the $\Ext$-group calculations given in Lemma \ref{lem:Ext-calc for S_i}, we obtain a long exact sequence of the form
$$\cdots \to \oplus _p H ^{k +n_1-n_G-1} (G ; R ) \maprt{\delta} H^k (G ; R)
\maprt{\gamma} \oplus_{i=0}^{p} H ^{k+n_1-n_{H_i}} (G ; R)
$$ $$ \to \oplus _p H ^{k +n_1-n_G} (G; R) \maprt{\delta}
H ^{k+1} (G; R) \to \cdots$$ where $k > d-n_1$. We claim that the map
$\gamma$ is injective. Observe that if $\gamma = \oplus \gamma _i$,
then for each $i$, the map $\gamma _i$ can be defined as
multiplication with some cohomology class $u_i$. To see this observe that $\gamma$ is the map induced by the differential $$d_{n_1-n_M+1} \colon \Ext _{\RG _G } ^k ( H_{n_1} (\CC ) , I_1 R ) \to \Ext _{\RG _G } ^{k +n_1-n_{M} +1} (H_{n_M} (\CC) , I_1 R)$$ on the hypercohomology spectral sequence given at (\ref{eqn:hyperSS}).  This spectral sequence has an $\Ext _{\RG} ^* (I_1 R, I_1 R)$-module structure, where the multiplication is given by the Yoneda product, defined by splicing the corresponding extensions (see \cite[Section 4]{benson-carlson2}).

Under the isomorphisms given in Lemma \ref{lem:Ext-calc for S_i}, the differential $d_{n_1-n_M +1}$ becomes a map $H^k (G, R) \to \oplus _i H^{k + n_1-n_{H_i} } (G, R)$ and the Yoneda product of Ext-groups is the same as the usual cup product multiplication in group cohomology under the canonical isomorphism $\Ext _{\RG_G} ^{m} (I_1 R, I_1 R)\cong H^m (G, R )$ (for comparison of different products on group cohomology see \cite[Proposition 4.3.5]{carlson1}). So we can conclude that $\gamma _i$ is the map defined by multiplication (the usual cup product) 
with a cohomology class $u_i \in H^{n_1-n_{H_i}} (G, R)$.  

 For $p=2$,   the cohomology ring $H^* (G,R)$ is isomorphic to a polynomial algebra $R[t_1, t_2]$ with $\deg t_i=1$ for $i=1,2$. Since there are no nonzero divisors in a polynomial algebra,  the map $\gamma$ is either injective or the zero map.  

For $p$ an odd prime, the cohomology ring $H^* (G,
R)$ is isomorphic to the tensor product of an exterior algebra with a polynomial algebra $$\Lambda _R (a_1, a_2)\otimes R[x_1, x_2]$$
where $\deg a_i=1$ and $\deg x_i=2$, and 
the nonzero divisors of this ring are multiples of $a_i$ or $a_j$. By Lemma \ref{lem:Case (i)},   $\deg u_i = n_1 -n_{H_i}\equiv 0 \pmod 2$, so the map $\gamma$ is either injective, or each $u_i$ must be a multiple of $a_1a_2$.

  The assumption that $\gamma$ is not injective implies that
 the entire spectral sequence restricted to some $H_i\cong \cy p$,
with $\Res_{H_i}^G(a_1a_2) = 0$,  will
result in a spectral sequence which collapses. 
This is because $\Res ^G _{H_i} I_G R=0$ and
$\Res ^G _{H_i} I_{H_j}R=0 $ if $i\neq j$. 
But the cohomology
$H^*(\cy p; R)$ is not finite-dimensional, and the
restriction of $\CC$ to a proper subgroup is still a finite
projective chain complex, so this gives a contradiction. Hence, we can conclude that
$\gamma $ is injective.

The fact that $\gamma$ is injective gives 
 a short exact sequence of the form 
$$0\to  H^k (G ; R)
\maprt{\gamma} \oplus_{i=0}^{p} H ^{k+n_1-n_{H_i}} (G ; R)
 \to \oplus _p H ^{k +n_1-n_G} (G; R) \to 0,$$
 for every $k> d-n_1$. 
Since $\dim _R H^m (G; R)=m+1$, 
we obtain $$ (k+1)+ p (k +
n_1-n_G+1)= \sum _{i=0} ^{p} (k+n_1-n_{H_i}+1).$$ Cancelling the
$(k+1)$'s and grouping the terms in a different way gives the desired
equality. 
\end{proof}

The next part uses the same spectral sequence, but the details are much simpler.
\begin{lemma}[Borel-Smith, part (iii)]\label{lem:Case (iii)} Let $G=\cy 4$, let $R = \cy 2$,  and let $\CC$ be a finite projective $\RG _G$-complex which is an $R$-homology $\un{n}$-sphere.
If $1\snor K \snor G$ with $K \cong \cy{2}$, then
$\un{n}(1)-\un{n}(K)$ is even.
\end{lemma}
\begin{proof} We consider the spectral sequence
$$E_2 ^{s,t} =\Ext ^s _{\RG _G} (H_t (\CC ), M),  $$
with $M = I_1R$, which converges to $H^{s+t} (\Hom _{\RG _G} (\CC , M ))$. 
Write $n_1 = \un{n}(1)$, $n_K = \un{n}(K)$, and $n_G = \un{n}(G)$.
Once again, the fact that 
$H^k(\CC;M)$ is zero in large dimensions $k > d = \dim \CC(1)$ gives rise to a long exact sequence
$$
\cdots \to \Ext _{\RG _G} ^{k +n_1-n_G+1} (I_G R, M) \maprt{\delta} \Ext^k
_{\RG_G} (I_1 R, M) \maprt{\gamma} \Ext _{\RG _G}
^{k+n_1-n_K+1} (I_{K} R , M)$$
$$ \to \Ext _{\RG _G} ^{k
+n_1-n_G+2} (I_G R, M) \maprt{\delta} \Ext _{\RG _G} ^{k+1} (I_1R, M)
\to \cdots
$$
The analogue of Lemma \ref{lem:Ext-calc for S_i} is easier in this case. We obtain
$$\Ext _{\RG _G}^{m} (I_{K} R , I_1R ) \cong \Ext _{\RG _G}^{m-1} (I_1 R , I_1R
)\cong H^{m-1} (G ; R)$$ for every $m\geq 1$, and 
$\Ext _{\RG _G}^{m} (I_{G} R , I_1R ) = 0$ for every $m \geq 0$. 
The vanishing result follows from the short exact sequence 
$$0 \to J_KR \to \un{R} \to I_GR \to 0$$
and the fact that $\Ext _{\RG _G}^{m} (J_K R , I_1R ) = 0$, for $m \geq 0$, since $J_KR$ has a projective resolution consisting of modules of the form $E_K P$.
On substituting these values into the long exact sequence, we obtain an isomorphism
$$\gamma\colon H^k(G; R) \cong H^{k + n_1 - n_K}(G; R)$$
induced by cup product, for all large $k$.
Since the cohomology ring $H^*(G;R)$ modulo nilpotent elements is generated by a $2$-dimensional class, it follows  that $n_1 - n_K$ must be even. 
\end{proof}

\begin{lemma}[Borel-Smith, part (iv)]\label{lem:Case (iv)}
Let $G=Q_8$, let $R = \cy 2$,  and let $\CC$ be a finite projective $\RG _G$-complex which is an $R$-homology $\un{n}$-sphere.
If $1\snor K \snor G$ with $K \cong \cy{2}$, then
$\un{n}(1)-\un{n}(K)$ is divisible by 4.
\end{lemma}
\begin{proof}
This time we have three index 2 
normal subgroups $H_1, H_2, H_3$,  each isomorphic to $\cy 4$. 
Write $n_1 = \un{n}(1)$, $n_K = \un{n}(K)$, $n_{H_i} = \un{n}(H_i)$, for $1\leq i \leq 3$, and $n_G = \un{n}(G)$. We again consider the spectral sequence
$$E_2 ^{s,t} =\Ext ^s _{\RG _G} (H_t (\CC ), M),  $$
with $M = I_1R$, which converges to $H^{s+t} (\Hom _{\RG _G} (\CC , M ))$. 
The exact sequences 
$$0 \to N \to \un{R} \to I_GR \to 0$$
and
$$0 \to (J_KR)^2 \to \oplus_i J_{H_i}R \to N \to 0$$
 lead to the calculation
$$\Ext _{\RG _G}^{m} (I_{G} R , I_1R ) =0$$
for every $m \geq 0$. The exact sequence
 $$0 \to I_1R \to J_KR \to I_K R \to 0$$
 implies that $\Ext _{\RG _G}^{m} (I_{K} R , I_1R ) = H^{m-1}(G;R)$, for $m \geq 1$. Finally, the exact sequences
 $$0 \to J_KR \to J_{H_i}R \to I_{H_i}R \to 0$$
 show that
 $\Ext _{\RG _G}^{m} (I_{H_i} R , I_1R ) = 0$, for $m \geq 0$ and $1\leq i \leq 3$.

 As a result of these calculations, we again obtain a 3-line spectral sequence with corresponding long exact sequence  $$
\cdots \to \Ext _{\RG _G} ^{k +n_1-n_G+1} (I_G R, M) \maprt{\delta} \Ext^k
_{\RG_G} (I_1 R, M) \maprt{\gamma} \Ext _{\RG _G}
^{k+n_1-n_K+1} (I_{K} R , M)$$
$$ \to \Ext _{\RG _G} ^{k
+n_1-n_G+2} (I_G R, M) \maprt{\delta} \Ext _{\RG _G} ^{k+1} (I_1R, M)
\to \cdots
$$
in all large dimensions $k > d$. By the vanishing result above,  the map 
$$\gamma\colon H^k (G;R) \cong H^{k+n_1 -n_K}(G;R)$$
is an isomorphism induced by cup product.
 Since the cohomology ring $H^*(G;R)$ modulo nilpotent elements is generated by a $4$-dimensional class, it follows  that $n_1 - n_K$ is divisible by 4. 
\end{proof}

\begin{remark}
 The fact that the dimension function of  an algebraic $\un{n}$-homology sphere satisfies the Borel-Smith conditions suggests that more of the classical results on finite group actions on spheres might hold for finite projective chain complexes over a suitable orbit category. For example, one could ask for an algebraic version of the results of Dotzel-Hamrick  \cite{dotzel-hamrick1}
on $p$-groups. Other potential applications of algebraic models to finite group actions are outlined in \cite{grodal-smith1}.
\end{remark}

\begin{example}\label{ex:Qdp}
An important test case for groups acting  on spheres, or on products of spheres \cite{adem-smith1}, is the rank two group $\Qd(p)=(\cy p \times \cy
p)\rtimes SL_2(p)$. At present, it is not known whether $\Qd(p)$ can act freely on a product of  two spheres, but  \" Unl\" u  \cite{unlu1} showed that 
$\Qd(p)$ does not act on a finite complex homotopy equivalent to a sphere with rank one isotropy.

We can apply the Borel-Smith conditions prove an algebraic version of this result. 

\begin{proposition}\label{pro:algQdp}
Let $p$ be an odd prime, $G=\Qd(p)$, $R=\bbZ /p$, and $\cF$ be   the family of all subgroups $H\leq G$ such that $\rk_p (H)\leq 1$. Let $\un{n}$ be a super class function with $\un{n}(1) \geq 0$.
Then, there exists no finite projective chain complex $\CC$ over $\RG _G$ which is an $R$-homology $\un{n}$-sphere.
\end{proposition}

\begin{proof} We can extend the family $\cF$ to the family $\Sub$ of all subgroups of $G$ by taking $\bC (H)=0$ for all subgroups such that $H \not \in \cF$. For these subgroups, we take $\un{n} (H)=-1$. Observe that by Theorem C, 
the dimension function $\un{n}: \Sub \to \bbZ$ satisfies the Borel-Smith conditions at the prime $p$. 

Now the rest of the argument follows as  in \"Unl\"u  \cite[Theorem 3.3]{unlu1}. Let $P$ be a Sylow $p$-subgroup of $\Qd(p)$. The group $P$ is
isomorphic to the extra-special $p$-group of order $p^3$ and
exponent $p$. If $Z(P)$ is the center of $P$, then the quotient
group $P/Z(P)$ is isomorphic to $\cy{p} \times \cy{p}$. Applying the
Borel-Smith condition (ii) for this quotient, we get $\un {n}(Z(P))=-1$. In $G$, it is possible to find two Sylow $p$-subgroups $P_1 $ and
$P_2$ such that $E=P_1 \cap P_2 \cong \cy{p} \times \cy{p}$ and
$Z(P_1)$ and $Z(P_2)$ are distinct subgroups of order $p$ in $E$.
Two such Sylow $p$-subgroups can be given as $ P_i = (\cy{p} \times
\cy{p} ) \rtimes \la A_i \ra $ for $i=1,2$ where
$$ A_1=\left ( \begin{matrix}  1 & 1 \\
0 & 1  \end{matrix}\right )\quad \quad A_2=\left (\begin{matrix}  1 & 0 \\
1 & 1  \end{matrix}\right )
$$
By the above argument,  $\un{n} (Z(P_i))= -1$, and non-central
$p$-subgroups in $E$ are conjugate to each other. So, we obtain that
$\un{n} (K)=-1$ for every subgroup $K$ of order $p$ in $E$. By the
Borel-Smith conditions applied to $E$, we get $\un{n} (1 )=-1$, contradicting our assumption on $\un{n}$.
\end{proof}

\end{example}

\providecommand{\bysame}{\leavevmode\hbox to3em{\hrulefill}\thinspace}
\providecommand{\MR}{\relax\ifhmode\unskip\space\fi MR }
\providecommand{\MRhref}[2]{%
  \href{http://www.ams.org/mathscinet-getitem?mr=#1}{#2}
}
\providecommand{\href}[2]{#2}

\end{document}